\documentclass[11pt,a4paper,oneside,english]{amsart}

\usepackage{geometry}
\usepackage{xcolor}  % 引入xcolor包，使超链接有颜色

\usepackage{lineno}
\usepackage[english]{babel}
\usepackage{amssymb,enumerate,bbm,amsmath}
\usepackage{tikz}
\usepackage{indentfirst}
\numberwithin{equation}{section}
\newtheorem{theorem}{Theorem}[section]

\newtheorem{lemma}[theorem]{Lemma}

\newtheorem{conjecture}[theorem]{Conjecture}
\newtheorem{remark}[theorem]{Remark}

\usepackage{hyperref} %生成引用链接，注：该宏包可能与其他宏包冲突，故放在所有引用的宏包之后
\usepackage{cleveref} %实现图片和表格、公式的引用，cleveref包必须放在hyperref包之后

\usepackage{verbatim}
\usepackage{color}
\hypersetup{colorlinks,linkcolor={blue},citecolor={blue},urlcolor={blue}}
\allowdisplaybreaks %允许长公式换行
\begin{document}

\title[]{Classification of closed minimal hypersurfaces with constant scalar curvature in $\mathbb{S}^5$}

\author{Chengchao He}
\address{Center of Mathematical Sciences, Zhejiang University, Hangzhou,
310027, People's Republic of China}
\email{22335022@zju.edu.cn}

\author{Hongwei Xu}
\address{Center of Mathematical Sciences, Zhejiang University, Hangzhou,
310027, People's Republic of China}
\email{xuhw@zju.edu.cn}

\author{Entao Zhao}
\address{Center of Mathematical Sciences, Zhejiang University, Hangzhou,
310027, People's Republic of China}
\email{zhaoet@zju.edu.cn}


\begin{abstract}
      In this paper, we prove that any closed minimal hypersurface $M^4$ in the $5$-dimensional unit sphere $\mathbb{S}^5$ with constant scalar curvature and constant $3$-th mean curvature must be isoparametric. To be precise, $M^4$ is either an equatorial 4-sphere, a product of spheres $\mathbb{S} ^{2}(\frac{\sqrt{2}}{2}) \times \mathbb{S} ^{2}(\frac{\sqrt{2}}{2})$ or $\mathbb{S} ^{1}(\frac{1}{2}) \times \mathbb{S} ^{3}(\frac{\sqrt{3}}{2})$,  or a Cartan's minimal hypersurface. In particular, the value of the squared norm of the second fundamental form $S$ can only be 0, 4, or 12. This result strongly supports Chern's conjecture.
\end{abstract}

{\maketitle}

\tableofcontents

\section{Introduction}

In 1968, Simons {\cite{J.Simons1968}} proved a famous rigidity theorem concerning closed minimal hypersurfaces in the unit sphere.

\begin{theorem}  \label{thm1.1}
    \textup{(J.Simons {\cite{J.Simons1968}})} Let \( M^n \) be an \( n \)-dimensional closed minimal hypersurface in the unit sphere \( \mathbb{S}^{n+1} \). Denote by \( S \) the squared length of the second fundamental form of \( M^n \). Then the following inequality holds:
    \begin{equation*}
        \int_{M^n} S \left( S - n \right) dM \geqslant 0.
    \end{equation*}
    
    In particular, if \( 0 \leqslant S \leqslant n \), then either $S\equiv0$ and $M^n$ is the totally geodesic sphere in \( \mathbb{S}^{n+1} \), or \(S\equiv n\).
\end{theorem} 

 Later, Chern-do Cramo-Kobayashi {\cite{Chern}} and Lawson {\cite{Lawson}} independently proved that if \(S\equiv n\), then $M^n$ is one of the Clifford minimal hypersurfaces \( \mathbb{S}^k\Big(\sqrt{\frac{k}{n}}\Big) \times \mathbb{S}^{n-k}\Big(\sqrt{\frac{n-k}{n}}\Big) \) in \( \mathbb{S}^{n+1} \), \( 1 \leqslant k \leqslant n-1 \). 
Based on these works, Chern {\cite{Chern-lecture,Chern}} proposed the following famous conjecture.
\begin{conjecture} \label{Chern's Conjecture}
    \textup{(Chern's Conjecture)} Let \(M^{n}\) be a closed immersed minimal hypersurface of the unit sphere \(\mathbb{S}^{n+1}\) with constant scalar curvature. Then for each \(n\), the set of all possible values for $S$ should be discrete.
\end{conjecture}

\begin{remark} 
    Theorem \ref{thm1.1} gave the first gap of $S$ for  Conjecture \ref{Chern's Conjecture}.
    By the Gauss equation, for minimal hypersurfaces $S\equiv{\rm constant}$ if and only if scalar curvature $R\equiv{\rm constant}$. 
    Consequently, Conjecture \ref{Chern's Conjecture} is equivalent to the statement that the possible values of $R$ form a discrete set. 
\end{remark}

The significance of Chern's conjecture has been highly recognized in a series of notable works. 
In 1982, Yau {\cite{Yau}} listed the aforementioned Chern's conjecture as one of the 120 unsolved world-famous geometric problems. 
In 1983, Peng-Terng {\cite{Peng & Terng1,Peng & Terng2}} made the first breakthrough towards Chern's conjecture, proving a second gap theorem that if $S > n$, then $S > n + \frac{1}{12n}$. 
Moreover, for $n = 3$, they proved that $S \geqslant 6$ provided $S > 3$. 
In 1993, Chang {\cite{Chang1}} completed the proof of Chern's conjecture for $n = 3$. 
Furthermore, based on progress on the conjecture made by Almeida-Brito {\cite{de Almeida}}, Chang {\cite{Chang2}} proved that the closed hypersurface $M^3$ of $\mathbb{S}^4$ with constant mean curvature and constant scalar curvature is also isoparametric. 
For high dimensions, there have been several estimates on the second gap, see e.g. 
Yang-Cheng {\cite{Yang & Cheng 1, Yang & Cheng 2,Yang & Cheng 3}}  and Suh-Yang {\cite{Suh & Yang}}. Up to now, the second gap has been improved to $\frac{3}{7}n$, which is obtained in {\cite{Suh & Yang}}. The following problem is still open.

\bigskip

\noindent\textbf{Second gap problem.} Let $M^{n}(n>3)$ be a closed minimal hypersurface in $\mathbb{S}^{n+1}$ with constant scalar curvature. Assume $S>n$. Does there hold $S\geqslant 2n$?

\bigskip

Currently, all known closed minimal hypersurfaces with constant scalar curvature in the sphere are isoparametric minimal hypersurfaces, which have constant principal curvatures. In 1986, Verstraelen et al. {\cite{Scherfner,Verstraelen}} proposed the following stronger version of Chern's conjecture.

\begin{conjecture}
    Let \(M^{n}\) be a closed immersed minimal hypersurface of the unit sphere \(\mathbb{S}^{n+1}\) with constant scalar curvature. Then \(M^{n}\) is isoparametric.
\end{conjecture}

As was mentioned, the conjecture is true for $n=3$ based on the work of Chang {\cite{Chang1}} and Peng-Terng {\cite{Peng & Terng1}}. For the case $n=4$, building on the results of Lusala-Scherfner-Sousa {\cite{Lusala 1,Lusala 2}}, Deng-Gu-Wei {\cite{Deng}} proved the following theorem.

\begin{theorem}  \label{thm-Deng}
    \textup{(Deng-Gu-Wei {\cite{Deng}})} Any closed minimal Willmore hypersurface $M^4$ in $\mathbb{S}^5$ with constant scalar
    curvature must be isoparametric. To be precise, $M^4$ is either an equatorial $4$-sphere, a product of spheres $\mathbb{S} ^{2}(\frac{\sqrt{2}}{2}) \times \mathbb{S} ^{2}(\frac{\sqrt{2}}{2})$ or a Cartan's minimal hypersurface.
    In particular, $S$ can only be $0,\ 4,\ 12$.
\end{theorem}

By computation, Theorem \ref{thm-Deng} can be stated equivalently as follows: any closed minimal hypersurface $M^4$ in $\mathbb{S}^5$ with vanishing 3-th mean curvature $H_3$ and constant scalar curvature must be isoparametric. Additionally, Tang-Yang {\cite{Tang-Yang}} proved that if the number of distinct principal curvatures is constant, then any closed minimal hypersurface $M^4$ in $\mathbb{S}^5$ with constant 3-th mean curvature $H_3$ and nonnegative constant scalar curvature must be isoparametric. For an $n$-dimensional closed hypersurface $M^n$ in $\mathbb{S}^{n+1}$, set  $f_k=\sum_{i=1}^{n} \lambda_{i}^{k}$ for $k=1,2,\cdots,n-1$, where $\lambda_1, \lambda_2,\cdots,\lambda_n$ are principal curvatures. Tang-Yan {\cite{Tang-Yan}} proved the following theorem.

\begin{theorem}
    \textup{(Tang-Yan {\cite{Tang-Yan}})} Let $M^n(n \geqslant 4)$ be a closed hypersurface in the unit sphere $\mathbb{S}^{n+1}$. 
    If the scalar curvature is nonnegative and $f_k\equiv{\rm constant}$ for $k=1,2,\cdots,n-1$, then $M^n$ is isoparametric.
\end{theorem}

In particular, a closed minimal hypersurface $M^4$ in $\mathbb{S}^5$ with constant 3-th mean curvature and nonnegative constant scalar curvature is isoparametric (see also {\cite{Xiao Ling}}). Furthermore, if the number of distinct principal curvatures is constant, then it was proved by Cheng-Li {\cite{Tongzhu Li}} that it is not necessary to assume that $M^4$ has nonnegative scalar curvature. In fact, they proved the following theorem.

\begin{theorem}
    \textup{(Cheng-Li {\cite{Tongzhu Li}})} Let $X: M^4 \to \mathbb{S}^5$ be a closed minimal hypersurface with constant squared length of the second fundamental form $S$ in the $5$-dimensional sphere. If the $3$-mean curvature $H_3$ and the number of distinct principal curvatures are constant, then $M^4$ is isoparametric and $S$ can only be $0,\ 4,\ or\ 12$.
\end{theorem}

In this paper, we will prove the following theorem, which strongly supports Chern's conjecture:

\begin{theorem}  \label{mthm}
    Any closed minimal hypersurface $M^4$ of $\mathbb{S} ^5$ with constant scalar curvature and constant $3$-th mean curvature must be isoparametric. To be precise, $M^4$ is either an equatorial $4$-sphere, a product of spheres $\mathbb{S} ^{2}(\frac{\sqrt{2}}{2}) \times \mathbb{S} ^{2}(\frac{\sqrt{2}}{2})$ or $\mathbb{S} ^{1}(\frac{1}{2}) \times \mathbb{S} ^{3}(\frac{\sqrt{3}}{2})$, or a Cartan's minimal hypersurface.  
    In particular, the value of the squared norm of the second fundamental form $S$ can only be $0,\ 4$, or $12$.
\end{theorem} 

\begin{remark}
 In Theorem \ref{mthm}, there is no assumption on the nonnegativity of the scalar curvature or on the number of distinct principal curvatures. Theorem \ref{mthm} is also a generalization of Deng-Gu-Wei's theorem from $H_3=0$ to $H_3$=constant.
\end{remark}  
\begin{remark} If we do not require $M^n$ to have constant scalar curvature, then the problem becomes more difficult. 
Peng-Terng {\cite{Peng & Terng1,Peng & Terng2}} proved that there exists a positive constant $\delta(n)$ depending only on $n$, such that if $n \leqslant S \leqslant n + \delta(n)$, $n \leqslant 5$, then $S \equiv n$. 
Later, Wei-Xu {\cite{Wei & Xu}} extended the result to $n = 6, 7$ and Zhang {\cite{Q.Zhang}} promoted it to $n \leqslant 8$.  Finally, Ding-Xin {\cite{Ding & Xin}} extended it to all dimensions by establishing a new estimate; in particular, they showed that if the dimension is $n \geqslant 6$, then the pinching constant $\delta(n) = \frac{n}{23}$. After that, Xu-Xu {\cite{Xu & Xu}} improved it to $\delta(n) = \frac{n}{22}$ and Lei-Xu-Xu {\cite{Lei & Xu & Xu}} showed $\delta(n) = \frac{n}{18}$.
\end{remark}

The paper is arranged as follows. In section 2, we will review basic formulas for closed minimal hypersurfaces in $\mathbb{S}^{n+1}$. In section 3, we will consider the case where there is one point in $M^4$ with two distinct principal curvatures and prove that it is the Clifford torus $\mathbb{S} ^{2}(\frac{\sqrt{2}}{2}) \times \mathbb{S} ^{2}(\frac{\sqrt{2}}{2})$ or $\mathbb{S} ^{1}(\frac{1}{2}) \times \mathbb{S} ^{3}(\frac{\sqrt{3}}{2})$. Finally, we will prove our main theorem in section 4.

\section{Preliminaries}

In this section, we recall basic formulas for closed minimal hypersurfaces in $\mathbb{S}^{n+1}$, which can be found in {\cite{Peng & Terng1,Peng & Terng2}}.

Let \(X:M^{n}\to\mathbb{S}^{n+1}\) be an \(n\)-dimensional immersed hypersurface in an \((n+1)\)-dimensional sphere \(\mathbb{S}^{n+1}\). Let $g$ denote the induced metric on $M^n$. For any point \(p\in M^{n}\), we choose a local orthonormal frame \(\{e_{1},\cdots,e_{n},e_{n+1}\}\) around \(p\) such that \(e_{1},\cdots,e_{n}\) are tangential to \(M^{n}\) and \(e_{n+1}\) is normal to \(M^{n}\). Let \(\{\omega_{1},\cdots,\omega_{n},\omega_{n+1}\}\) be the dual coframe and \(\{\omega_{ij}\mid 1\leqslant i,j\leqslant n\}\) be the connection \(1\)-forms. In this section, we make the following convention on the range of indices:
\[1\leqslant i,j,k,\cdots\leqslant n.\]
Then the structure equations of \(M^{n}\) are given by
\begin{equation}
\begin{split}
& d\omega_{i}=\sum_{j}\omega_{ij}\wedge\omega_{j},\quad \omega_{ij}+\omega_{ji}=0,\\
& d\omega_{ij}=\sum_{k}\omega_{ik}\wedge\omega_{kj}-\frac{1}{2}\sum_{k,l}R_{ijkl}\omega_{k}\wedge\omega_{l},
\end{split}
\end{equation}
where \(R_{ijkl}\) is the Riemannian curvature tensor of the induced metric on \(M^{n}\).

Let \(h=\sum_{ij}h_{ij}\omega_{i}\otimes\omega_{j}\) denote the second fundamental form, and \(H=\frac{1}{n}\sum_{i}h_{ii}\) the mean curvature. The Gauss equation is
\begin{equation} \label{Gauss equation}
R_{ijkl}=\delta_{ik}\delta_{jl}-\delta_{il}\delta_{jk}+h_{ik}h_{jl}-h_{il}h_{jk},
\end{equation}
and the Codazzi equation is
\begin{equation}
h_{ijk}=h_{ikj},
\end{equation}
where the covariant derivative of the second fundamental form is defined by
\begin{equation}\label{h_ijk}
\sum_{m}h_{ijm}\omega_{m}=dh_{ij}+\sum_{m}h_{mj}\omega_{mi}+\sum_{m}h_{im}\omega_{mj}.
\end{equation}

The second covariant derivative of the second fundamental form is defined by
\[
\sum_{m}h_{ijkm}\omega_{m}=dh_{ijk}+\sum_{m}h_{mjk}\omega_{mi}+\sum_{m}h_{imk}\omega_{mj}+\sum_{m}h_{ijm}\omega_{mk}.
\]
Thus we have the following Ricci identity:
\begin{equation}
h_{ijkl}-h_{ijlk}=\sum_{m}h_{mj}R_{mikl}+\sum_{m}h_{im}R_{mjkl}.
\end{equation}

By the Gauss equation \eqref{Gauss equation}, we obtain the Ricci curvature tensor $R_{ij}$ and the scalar curvature $R$ of the hypersurface:
\begin{equation}
\begin{split}
R_{ij} &= (n-1)\delta_{ij} + nHh_{ij} - \sum_{m}h_{im}h_{mj}, \\
R &= n(n-1) + n^{2}H^{2} - S,
\end{split}
\end{equation}
where $S = \sum_{i,j}h_{ij}^{2}$ is the square norm of the second fundamental form.

Define $f_3$ and $f_4$ by
$$f_3=\sum_{i,j,k}h_{ij}h_{jk}h_{ki},\ \ f_4=\sum_{i,j,k,l}h_{ij}h_{jk}h_{kl}h_{li}.$$

For an arbitrary fixed point $p\in M^n$, we take an orthonormal frame such that $h_{ij}=\lambda_{i}\delta_{ij}$ at $p$, for all $i,j$. 
Then at this point $p$, we have 
$$H=\frac{1}{n}\sum_{i=1}^{n}\lambda_i,\ \ S=\sum_{i=1}^{n}\lambda_{i}^2,\ \ f_3=\sum_{i=1}^{n}\lambda_{i}^3,\ \ f_4=\sum_{i=1}^{n}\lambda_{i}^4.$$
The Gauss-Kronecker curvature is given by
$$K=\frac{\textup{det}(h)}{\textup{det}(g)}=\prod_{i=1}^{n}\lambda_{i}.$$
Define $\mathcal{A}$ and $\mathcal{B} $ by
$$\mathcal{A} =\sum_{i,j,k}h_{ijk}^{2}\lambda_{i}^2,\ \ \mathcal{B} =\sum_{i,j,k}h_{ijk}^{2}\lambda_{i}\lambda_{j}.$$

For closed minimal hypersurfaces with constant scalar curvature, the following formulas are obtained through straightforward computation:
\begin{equation}
    \sum_{i,j,k} h^{2}_{ijk} = S(S - n),
\end{equation}
\begin{equation}
    \sum_{i,j,k,l} h^{2}_{ijkl} = (S - 2n - 3)S(S - n) + 3(\mathcal{A}  - 2\mathcal{B} ),
\end{equation}
\begin{equation}
    \Delta f_{3} = 3(n - S)f_{3} + 6h_{ijk}h_{ijl}h_{kl},
\end{equation}
\begin{equation}
    \Delta f_{4} = 4\big((n - S)f_{4} + 2\mathcal{A}  + \mathcal{B} \big),
\end{equation}
\begin{equation}
    \mathcal{A}  - 2\mathcal{B}  = Sf_{4} - f_{3}^{2} - S^{2}.
\end{equation}

Let \(\sigma_r : \mathbb{R}^n \to \mathbb{R}\) be the elementary symmetric functions defined by
\[
\sigma_r(\lambda_1, \dots, \lambda_n) = \sum_{i_1 < i_2 < \cdots < i_r} \lambda_{i_1} \lambda_{i_2} \cdots \lambda_{i_r},\ \ 1\leqslant r \leqslant n.
\]

\begin{lemma}
    \textup{(Newton's formula)} Let $p(\lambda) = (\lambda - \lambda_1)(\lambda - \lambda_2)\cdots (\lambda - \lambda_n)$ and 
    $f_k = \lambda_{1}^k + \lambda_{2}^k +\cdots + \lambda_{n}^k \,\,(k=1,2,\cdots)$. Then
    \[
    \begin{cases}
        &f_k - \sigma_{1}f_{k-1} + \cdots + (-1)^{k-1}\sigma_{k-1}f_{1} + (-1)^{k}k\sigma_k = 0,(1\leqslant k\leqslant n); \\[6pt]
        &f_k - \sigma_{1}f_{k-1} + \cdots + (-1)^{n}\sigma_{n}f_{k-n} = 0,(k>n).
    \end{cases}
    \]
\end{lemma}

Define the $r$-th mean curvature by
\[
H_r=\frac{1}{ \binom{n}{r} }\sigma_r, \ \ 1\leqslant r \leqslant n.
\]

When $n=4$, by Newton's formula, we have
\[
\begin{cases}
    &f_1 - \sigma_1 = 0, \\
    &f_2 - f_{1}\sigma_1 + 2\sigma_2 = 0, \\
    &f_3 - f_{2}\sigma_1 + f_{1}\sigma_2 - 3\sigma_3 = 0, \\
    &f_4 - f_{3}\sigma_1 + f_{2}\sigma_2 - f_{1}\sigma_3 + 4\sigma_4 = 0.
\end{cases}
\]

Since the hypersurface is assumed to be minimal and $f_{2}=S$, one has $H_3=\frac{1}{4}\sigma_3=\frac{1}{12}f_3$, and the characteristic polynomial of the matrix $(h_{ij})$ corresponding to the second fundamental form is given by
\begin{equation} \label{the characteristic polynomial}
    p(\lambda)=\lambda^4-\frac{S}{2} \lambda^2-\frac{f_3}{3}\lambda+K.
\end{equation}
Also, one has
\begin{equation}
    f_4 = \frac{S^2}{2} - 4K.
\end{equation}

Obviously, if $S\equiv0$, then $M^4$ is the totally geodesic great sphere $\mathbb{S}^{4}(1)$. 
Suppose from now that $S>0$ on $M^4$. 

\section{Two distinct principle curvatures at one point}

In this section, we will deal with the case where there exists one point that has two distinct principal curvatures. 
More precisely, we will prove the following result.

\begin{theorem}
    Let $M^4 \hookrightarrow \mathbb{S}^{5}(1)$ be a closed minimal hypersurface with constant scalar curvature and constant $3$-th mean curvature. 
    If there exists a point with two distinct principal curvatures, then $S=4$ and $M^4$ is the Clifford torus $\mathbb{S} ^{2}(\frac{\sqrt{2}}{2}) \times \mathbb{S} ^{2}(\frac{\sqrt{2}}{2})$
    or $\mathbb{S} ^{1}(\frac{1}{2}) \times \mathbb{S} ^{3}(\frac{\sqrt{3}}{2})$.
\end{theorem}

\begin{proof}
The existence of two distinct principal curvatures at a point can be divided into two cases: the first is that $p(\lambda)$ has two different double real roots, and the second is that $p(\lambda)$ has one triple real root and one simple real root.
\medskip

\textbf{Case \textrm{I}:} Two different double real roots.

In this case, we may assume that $\lambda_1=\lambda_2 < \lambda_3=\lambda_4$. Since $M$ is minimal, one has $\lambda_1=\lambda_2=-\lambda_3=-\lambda_4$. Obviously $f_3 \equiv 0$ as $f_3$ is constant.  By Theorem \ref{thm-Deng}, $S = 4$ and $M^4$ is the Clifford torus $\mathbb{S} ^{2}(\frac{\sqrt{2}}{2}) \times \mathbb{S} ^{2}(\frac{\sqrt{2}}{2})$.

\medskip

\textbf{Case \textrm{II}:} One triple real root and one simple real root.

In this case, we may assume $\lambda_1=\lambda_2=\lambda_3=a,\lambda_4=-3a.$

Since the mean curvature $H = 0$ and $S=constant$, after taking the covariant derivative, we have the following for $1 \leqslant k \leqslant 4$,
\[
\begin{cases}
	h_{11k}+h_{22k}+h_{33k}+h_{44k}=0, \\
	ah_{11k}+ah_{22k}+ah_{33k}-3ah_{44k}=0 .
\end{cases}
\]
This system yields
$$h_{11k}+h_{22k}+h_{33k}=0, \quad h_{44k}=0,\ \ \  1\leqslant k \leqslant 4.$$

Since $H=\sum\limits_{i} h_{ii}= 0$, we have $(H)_{mm}=0$ for $m=1,2,3,4$, i.e.,
$$\sum_{i} h_{iimm}=0.$$
This gives
\[
\begin{cases}
    h_{1111} + h_{2211} + h_{3311} + h_{4411} = 0, \\
    h_{1122} + h_{2222} + h_{3322} + h_{4422} = 0, \\
    h_{1133} + h_{2233} + h_{3333} + h_{4433} = 0, \\
    h_{1144} + h_{2244} + h_{3344} + h_{4444} = 0.
\end{cases}
\]

Since $S=\sum\limits_{i,j}h_{ij}^2$ is constant, $(S)_{mm}=0$ for $m=1,2,3,4$. So we have
\begin{equation} \label{derivative of S}
    \sum_{i,j}(h_{ii}h_{iimm}+h_{ijm}^2)=0.
\end{equation}
For $m = 1$:
\begin{align*} 
    \sum_{i,j}(h_{ii}h_{ii11}+h_{ij1}^2)
    =& \, a(h_{1111}+h_{2211}+h_{3311})-3ah_{4411}+h_{111}^2+h_{221}^2+h_{331}^2 \\[-0.5em]
     & +2(h_{112}^2+h_{113}^2+h_{114}^2+h_{123}^2+h_{124}^2+h_{134}^2) \\[0.5em]
    =& -4ah_{4411}+h_{111}^2+h_{221}^2+h_{331}^2\\[0.5em]
     & +2(h_{112}^2+h_{113}^2+h_{114}^2+h_{123}^2+h_{124}^2+h_{134}^2)\\[0.5em]
    =& \, 0,
\end{align*}
i.e.,
\begin{equation}  \label{S_11}
    h_{111}^2+h_{221}^2+h_{331}^2+2(h_{112}^2+h_{113}^2+h_{114}^2+h_{123}^2+h_{124}^2+h_{134}^2)=4ah_{4411}.
\end{equation}
For $m = 4$:
\begin{align*}
    \sum_{i,j}(h_{ii}h_{ii44}+h_{ij4}^2) 
    =& \, a(h_{1144}+h_{2244}+h_{3344})-3ah_{4444}\\[-0.5em]
     &+\left[h_{114}^2+h_{224}^2+h_{334}^2+2(h_{124}^2+h_{134}^2+h_{234}^2)\right] \\[0.5em]
    =& -4ah_{4444}+\left[h_{114}^2+h_{224}^2+h_{334}^2+2(h_{124}^2+h_{134}^2+h_{234}^2)\right]\\[0.5em]
    =& \, 0,
\end{align*}
i.e.,
\begin{equation}  \label{S_44}
    h_{114}^2+h_{224}^2+h_{334}^2+2(h_{124}^2+h_{134}^2+h_{234}^2)=4ah_{4444}.
\end{equation}

Similarly, since $f_3=\sum\limits_{i,j,k} h_{ij}h_{jk}h_{ki}$ is constant, $(f_3)_{mm}=0$  for $m=1,2,3,4$. So we have
\begin{equation}  \label{derivative of f_3}
    \sum_{i,j} (\lambda_{i}^2 h_{iimm}+2\lambda_i h_{ijm}^2)=0.
\end{equation}
For $m = 1$:
\begin{align*}
    \sum_{i,j} (\lambda_{i}^2 h_{ii11}+2\lambda_i h_{ij1}^2) 
    =& \, a^2(h_{1111}+h_{2211}+h_{3311})+9a^2h_{4411}\\[-0.5em]
     &+2a\sum_{i}(h_{11i}^2+h_{12i}^2+h_{13i}^2-3h_{14i}^2) \\
    =& \, 8a^2h_{4411}+2a\big[h_{111}^2+h_{221}^2+h_{331}^2\\[0.5em]
     &+2(h_{112}^2+h_{113}^2-h_{114}^2+h_{123}^2-h_{124}^2-h_{134}^2)\big]\\[0.5em]
    =& \, 0,
\end{align*}
i.e.,
\begin{equation}  \label{f_3 11}
    h_{111}^2+h_{221}^2+h_{331}^2+2(h_{112}^2+h_{113}^2-h_{114}^2+h_{123}^2-h_{124}^2-h_{134}^2)=-4ah_{4411}.
\end{equation}
For $m = 4$: 
\begin{align*}
    \sum_{i,j} (\lambda_{i}^2 h_{ii44}+2\lambda_i h_{ij4}^2)
    =& \, a^2(h_{1144}+h_{2244}+h_{3344})+9a^2h_{4444}\\[-0.5em]
     &+2a\sum_{i}(h_{14i}^2+h_{24i}^2+h_{34i}^2-3h_{44i}^2) \\
    =& \, 8a^2h_{4444}+2a\left[h_{114}^2+h_{224}^2+h_{334}^2+2(h_{124}^2+h_{134}^2+h_{234}^2)\right]\\[0.5em]
    =& \, 0,
\end{align*}
i.e.,
\begin{equation}  \label{f_3 44}
    h_{114}^2+h_{224}^2+h_{334}^2+2(h_{124}^2+h_{134}^2+h_{234}^2)=-4ah_{4444}.
\end{equation}

Adding equations \eqref{S_11} and \eqref{f_3 11} gives $$h_{111}^2+h_{221}^2+h_{331}^2+2(h_{112}^2+h_{113}^2+h_{123}^2)=0.$$ So
$$h_{111}=h_{221}=h_{331}=h_{112}=h_{113}=h_{123}=0.$$

Adding equations \eqref{S_44} and \eqref{f_3 44} gives $$h_{114}^2+h_{224}^2+h_{334}^2+2(h_{124}^2+h_{134}^2+h_{234}^2)=0.$$  So
$$h_{114}=h_{224}=h_{334}=h_{124}=h_{134}=h_{234}=0.$$

Now, only $h_{222},h_{223},h_{332},h_{333}$ are possibly non-zero. We show that these also vanish.

For $m = 2$, equations \eqref{derivative of S} and \eqref{derivative of f_3} become
\[
\begin{cases}
    h_{222}^2+2h_{223}^2+h_{332}^2 =4ah_{4422}, \\
    h_{222}^2+2h_{223}^2+h_{332}^2 =-4ah_{4422}.
\end{cases}
\]
So $h_{222}^2+2h_{223}^2+h_{332}^2=0$, and hence
$$h_{222}=h_{223}=h_{332}=0.$$

For $m = 3$, equations \eqref{derivative of S} and \eqref{derivative of f_3} become
\[
\begin{cases}
    h_{333}^2 =4ah_{4433}, \\
    h_{333}^2 =-4ah_{4433}.
\end{cases}
\]
So $h_{333}^2=0$, and hence
$$h_{333}=0.$$

Thus, we have shown that $h_{ijk} = 0$ for all $i, j, k = 1, 2, 3, 4$. 
Since $\sum_{i,j,k} h_{ijk}^2 = S(S - 4)$, we have $S(S - 4) = 0$. 
Given $S > 0$, it follows that $S = 4$ and $M^4$ is the Clifford torus $\mathbb{S} ^{1}(\frac{1}{2}) \times \mathbb{S} ^{3}(\frac{\sqrt{3}}{2})$.
\end{proof}

\begin{remark}
    This case has been dealt with in {\cite{Xiao Ling}}, but our proof is more direct. 
\end{remark}

\section{Proof of the main theorem}
In the previous section, we proved that if $M^4$ has two distinct principal curvatures at a point, then $S=4$ and $M^4$ is the Clifford torus $\mathbb{S} ^{2}(\frac{\sqrt{2}}{2}) \times \mathbb{S} ^{2}(\frac{\sqrt{2}}{2})$ or $\mathbb{S} ^{1}(\frac{1}{2}) \times \mathbb{S} ^{3}(\frac{\sqrt{3}}{2})$. When $M^4$ has four identical principal curvatures at a point, it corresponds to the trivial case where $S=0$ and $M^4$ is an equatorial 4-sphere. Therefore, in this section, we may assume that every point in $M^4$ has either three or four distinct principal curvatures. 

From now on, we also assume that $M^4$ is connected and oriented. 
Otherwise, we can discuss the situation on each connected component of $M^4$ or on the double covering of $M^4$.

For convenience, we first make some notations. 
Let us denote by $\lambda_{1}(p)\leqslant \lambda_{2}(p)\leqslant \lambda_{3}(p) \leqslant \lambda_{4}(p)$ the eigenvalues of $h(p)$ for each $p\in M^4$. 
For convenience of exposition, we will assume that $f_3 \geqslant 0$; otherwise, we consider $\tilde{\lambda}_i = -\lambda_i$ by choosing the normal vector field in the opposite direction.
Note that $\lambda_i$ is continuous for each $i=1,2,3,4$. 
The characteristic polynomial \eqref{the characteristic polynomial} of $h$ is important in our discussion:
\begin{equation*}
    p(\lambda) = \lambda^4 - \frac{S}{2}\lambda^2 - \frac{f_3}{3}\lambda + K.
\end{equation*}

We set
\begin{equation}
    p_{0}(\lambda) = \lambda^4 - \frac{S}{2}\lambda^2 - \frac{f_3}{3}\lambda.
\end{equation}

Clearly, $p_{0}(\lambda)$ is a polynomial of degree 4 with coefficients depending on $S,f_3$ and independent of $K$. 
Moreover, combining \eqref{the characteristic polynomial}, we have
\begin{equation}
    p(\lambda) = p_{0}(\lambda) + K.
\end{equation}

\subsection{Discussion on   \texorpdfstring{$\text{Im}K$}{ImK}}
$\ $
\medskip

As the Gauss curvature $K$ is a smooth function on the closed $M^4$, we can assume that the range of $K$ is 
$$\text{Im}K = [a_0,b_0], \quad  a_0 \leqslant b_0.$$

To further determine $\text{Im}K$, we study the graph of the polynomial $p_{0}(\lambda)$ which is independent of $K$. 

It is clear that the graph of $p_{0}(\lambda)$, as a quartic function, can only fall into one of the three cases. See the following Figures 1, 2, and 3 of the polynomial function $p_{0}(\lambda)$, where $\mu_1 \leqslant \mu_2 \leqslant \mu_3$ are all the extreme points of $p_{0}(\lambda)$.
\begin{center}
    \tikzset{every picture/.style={line width=0.75pt}} %set default line width to 0.75pt        

    \begin{tikzpicture}[x=0.75pt,y=0.75pt,yscale=-1,xscale=1]

    %Curve Lines [id:da14201825890381736] 
    \draw    (214,106.5) .. controls (255,229) and (328,237.5) .. (356,103.5) ;

    % Text Node
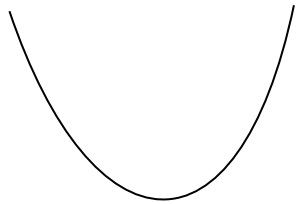
    \draw (191,220) node [anchor=north west][inner sep=0.75pt]   [align=left] {Figure 1. There is one extreme point};

\end{tikzpicture}

\end{center}

\medskip
\medskip

\begin{center}
    \tikzset{every picture/.style={line width=0.75pt}} %set default line width to 0.75pt        

    \begin{tikzpicture}[x=0.75pt,y=0.75pt,yscale=-1,xscale=1]

    %Curve Lines [id:da14201825890381736] 
    \draw    (185,74.5) .. controls (203,189.5) and (239,199.5) .. (284,132.5) ;
    %Curve Lines [id:da3225520100160256] 
    \draw    (284,132.5) .. controls (340,49.5) and (339,182.5) .. (375,192.5) ;
    %Curve Lines [id:da4219932812960274] 
    \draw    (375,192.5) .. controls (394,199.5) and (432,169.5) .. (446,80.5) ;
    %Straight Lines [id:da8582742998073106] 
    \draw  [dash pattern={on 4.5pt off 4.5pt}]  (316,105) -- (470,105) ;
    %Straight Lines [id:da058060599061717344] 
    \draw  [dash pattern={on 4.5pt off 4.5pt}]  (234,175) -- (472,175) ;
    %Straight Lines [id:da26387113964045517] 
     \draw  [dash pattern={on 4.5pt off 4.5pt}]  (381,195) -- (468,195) ;

    % Text Node
    \draw (144,220) node [anchor=north west][inner sep=0.75pt]   [align=left] {Figure 2. There are three extreme points and $\displaystyle p_{0}( \mu _{1}) \geqslant p_{0}( \mu _{3})$};
    % Text Node
    \draw (483,94.4) node [anchor=north west][inner sep=0.75pt]    {$p_{0}( \mu _{2})$};
    % Text Node
    \draw (483,162.4) node [anchor=north west][inner sep=0.75pt]    {$p_{0}( \mu _{1})$};
    % Text Node
    \draw (484,189.4) node [anchor=north west][inner sep=0.75pt]    {$p_{0}( \mu _{3})$};

    \end{tikzpicture}

\end{center}

\medskip

\begin{center}
    \tikzset{every picture/.style={line width=0.75pt}} %set default line width to 0.75pt        

    \begin{tikzpicture}[x=0.75pt,y=0.75pt,yscale=-1,xscale=1]

    %Curve Lines [id:da14201825890381736] 
    \draw    (178,63.5) .. controls (187.55,147.95) and (248,281.5) .. (293,143.5) .. controls (338,5.5) and (376,327.5) .. (422,65.5) ;
    %Straight Lines [id:da8582742998073106] 
    \draw  [dash pattern={on 4.5pt off 4.5pt}]  (315,107) -- (469,107) ;
    %Straight Lines [id:da058060599061717344] 
    \draw  [dash pattern={on 4.5pt off 4.5pt}]  (243,204) -- (458,204) -- (470,204) ;
    %Straight Lines [id:da26387113964045517] 
    \draw  [dash pattern={on 4.5pt off 4.5pt}]  (378,172) -- (465,172) ;

    % Text Node
    \draw (144,220) node [anchor=north west][inner sep=0.75pt]   [align=left] {Figure 3. There are three extreme points and $\displaystyle p_{0}( \mu _{1}) \leqslant p_{0}( \mu _{3})$};
    % Text Node
    \draw (483,96.4) node [anchor=north west][inner sep=0.75pt]    {$p_{0}( \mu _{2})$};
    % Text Node
    \draw (484,195.4) node [anchor=north west][inner sep=0.75pt]    {$p_{0}( \mu _{1})$};
    % Text Node
    \draw (484,163.4) node [anchor=north west][inner sep=0.75pt]    {$p_{0}( \mu _{3})$};

    \end{tikzpicture}
\end{center}

Recalling the assumption stated at the beginning of this section, i.e., every point in $M^4$ has either three or four distinct principal curvatures, this is equivalent to that the equation
\begin{equation}  \label{the characteristic polynomial equation}
    p(\lambda) = p_{0}(\lambda) + K = 0
\end{equation}
has three or four distinct roots on $M^4$, which is also equivalent to that the graph of $p_{0}(\lambda)$ intersects the line $y=-K$ at three or four distinct points. Clearly, only Figures 2 and 3 satisfy the condition that the graph of $p_{0}(\lambda)$ intersects the line $y=-K$ at three or four distinct points, and obviously $\mu_1 < \mu_2 < \mu_3$.

Define
$$a := -p_{0}(\mu_2) \quad \text{and} \quad b := -\max\{p_{0}(\mu_1), p_{0}(\mu_3) \}.$$
Then $-b \leqslant -K \leqslant -a$, i.e.,$a \leqslant K \leqslant b$.
We obtain immediately that
$$\text{Im}K = [a_0,b_0]\subset [a,b].$$
So we need to consider four cases: (1) $a_{0} > a$, $b_{0} < b$; (2) $a_{0} = a$, $b_{0} < b$; (3) $a_{0} > a$, $b_{0} = b$; (4) $a_{0} = a$, $b_{0} = b$.

For the case (1), i.e., all the principal curvatures are distinct on $M^{4}$, it has already been completed by Cheng-Li {\cite{Tongzhu Li}}. 
So in the following, it is sufficient for us to deal with the case (4), and the other cases are analogous. 

In addition, if $K = b_0 = b$, then it can be seen from Figure 2 and 3 that $\lambda_1 = \lambda_2 < \lambda_3 < \lambda_4(\lambda_1 = \lambda_2 < 0)$ or $\lambda_1 < \lambda_2  < \lambda_3 = \lambda_4(\lambda_3 = \lambda_4 > 0)$. 
For the second case, one has 
\begin{align*}
    \lambda_1 &= -\lambda_4 - \sqrt{\frac{1}{2}S - 2\lambda_{4}^{2}}, \\
    \lambda_2 &= -\lambda_4 + \sqrt{\frac{1}{2}S - 2\lambda_{4}^{2}},
\end{align*}
which gives
$$f_3 = \sum_{i=1}^{4} \lambda_{i}^{3} = -6\lambda_4 (\frac{1}{2}S - 2\lambda_{4}^{2}) < 0,$$
a contradiction. So this case does not exist.

If $K = a_0 = a$, then $\lambda_1 < \lambda_2 = \lambda_3 < \lambda_4$, and 
\begin{align*}
  \lambda_1 &= -\lambda_2 - \sqrt{\frac{1}{2}S - 2\lambda_{2}^{2}}, \\
  \lambda_4 &= -\lambda_2 + \sqrt{\frac{1}{2}S - 2\lambda_{2}^{2}}, \\
  f_3 &= \sum_{i=1}^{4} \lambda_{i}^{3} = -6\lambda_2 (\frac{1}{2}S - 2\lambda_{2}^{2}) \geqslant 0,
\end{align*}
which gives
$$\lambda_2 = \lambda_3 \leqslant 0.$$
If $\lambda_2 = \lambda_3 = 0$, then $f_3 = 0$ and by Theorem \ref{thm-Deng}, Theorem \ref{mthm} follows. Thus, we only need to consider the case that 
\begin{equation}
  \lambda_2 = \lambda_3 < 0
\end{equation}

From now on, we assume that $\text{Im}K = [a, b]$ with $a < b$, and $a, b$ are achieved on the points where $p_{0}(\lambda)$ achieves the local maximum values and the maximum of its local minimum values. In addition, based on the discussion above, we know that when $K = a$, then $\lambda_1 < \lambda_2 = \lambda_3 < 0 < \lambda_4$, and when $K = b$, then $\lambda_1 = \lambda_2 < \lambda_3 < \lambda_4(\lambda_1 = \lambda_2 < 0)$

We define
\begin{equation} \label{set classification}
    \begin{aligned}
        X :&= K^{-1}(a) \\
           &= \{p \in M^{4}: K(p) = a\} \\
           &= \{p \in M^{4}: \lambda_1(p) < \lambda_2(p) = \lambda_3(p) < 0 < \lambda_4(p)\}, \\
        Y :&= \{p \in M^{4}: a < K(p) < b\} \\
           &= \{p \in M^{4}: \lambda_1(p) < \lambda_2(p) < \lambda_3(p) < \lambda_4(p)\}, \\
        Z :&= K^{-1}(b) \\
           &= \{p \in M^{4}: K(p) = b\} \\
           &= \{p \in M^{4}: \lambda_1 (p) = \lambda_2(p) < \lambda_3(p) < \lambda_4(p)(\lambda_1(p) = \lambda_2(p) < 0)\}.
    \end{aligned}
\end{equation}
Obviously,
$$
M^{4} = X \cup Y \cup Z.$$
If $Y = \varnothing $, then $K$ equals $a$ or $b$, and every point on $M^4$ has three distinct principal curvatures. By {\cite{Tongzhu Li}}, $M^4$ is isoparametric. 
But from {\cite{Cartan}} and {\cite{Mnzner}}, we know this is a contradiction as it is well known from Cartan’s classification result that isoparametric hypersurfaces of $\mathbb{S}^{n+1}$ with three distinct principal curvatures do exist only if $n = 3, 6, 12, 24$. Thus there is no such hypersurface in $\mathbb{S}^5$.

From now on, we will assume $Y \neq \varnothing $.

\subsection{Structure equations on  \texorpdfstring{$Y$}{Y}}
$\ $

\medskip
 
Locally, we choose an oriented orthonormal frame field $\{e_1,e_2,e_3,e_4 \}$ on $M^4$. 
Let $\{\omega_1, \omega_2, \omega_3, \omega_4\}$ be the dual frame. We can choose a proper coordinate system on $Y$ such that $(U, (\omega_1, \omega_2, \omega_3, \omega_4))$ is admissible. 
Namely, $(U, (\omega_1, \omega_2, \omega_3, \omega_4))$ satisfies
\begin{itemize}
    \item $(\omega_1, \omega_2, \omega_3, \omega_4)$ is a smooth orthonormal coframe field on an open subset $U$ of $Y$;
    \item $\omega_1 \wedge \omega_2 \wedge \omega_3 \wedge \omega_4 = \textup{vol}$, the volume form on $U$; 
    \item $h = \sum_{i=1}^{4} \lambda_i \omega_i \otimes \omega_i$.
\end{itemize}
Evidently, when $(U, (\omega_1, \omega_2, \omega_3, \omega_4))$ is admissible, the connection forms $\omega_{ij}$ on $U$ are uniquely determined and $h_{ij} = \lambda_i \delta_{ij}$.

We differentiate each $\lambda_i$ for $i=1,2,3,4$ to obtain the smooth 1-form $d\lambda_i$, which can be expressed by the metric form $\omega_k$ as
$$d\lambda_i = \sum_{j=1}^{4} \lambda_{ij} \omega_j,$$
where $\lambda_{ij}$'s are smooth functions on $U$. 
Besides, express the connection form $\omega_{ij}$ as
\begin{equation}  \label{w-ij}
    \omega_{ij} := \sum_{k=1}^{4} \varGamma_{ijk} \omega_k,
\end{equation}
where $\varGamma_{ijk} = \omega_{ij}(e_k)$ is the connection coefficient. 
Then it follows from \eqref{h_ijk} immediately that
\begin{align*}
    \sum_{k=1}^{4} h_{iik} \omega_k &= d\lambda_i = \sum_{k=1}^{4} \lambda_{ik} \omega_k, \quad \forall i=1,2,3,4, \\
    \sum_{k=1}^{4} h_{ijk} \omega_k &= (\lambda_i - \lambda_j) \omega_{ij} = (\lambda_i - \lambda_j) \sum_{k=1}^{4} \varGamma_{ijk} \omega_k, \quad \forall i \neq j.
\end{align*}
Therefore,
\begin{align}
    h_{iik} &= \lambda_{ik},  \label{h-iik} \\
    h_{ijk} &= (\lambda_i - \lambda_j) \varGamma_{ijk}, \quad \forall i \neq j. \label{h-ijk}
\end{align}

Differentiating the equations
\begin{equation}
\begin{cases}
\lambda_{1} + \lambda_{2} + \lambda_{3} + \lambda_{4} = 0, \\
\lambda_{1}^2 + \lambda_{2}^2 + \lambda_{3}^2 + \lambda_{4}^2 = S, \\
\lambda_{1}^3 + \lambda_{2}^3 + \lambda_{3}^3 + \lambda_{4}^3 = f_3, \\
\lambda_{1}^4 + \lambda_{2}^4 + \lambda_{3}^4 + \lambda_{4}^4 = f_4,
\end{cases}
\end{equation}
and noticing that $ f_4 = \frac{S^2}{2} - 4K$, we obtain that for each $j=1,2,3,4$,
\begin{equation}  \label{Vandermonde matrix}
\begin{pmatrix}
1 & 1 & 1 & 1 \\
\lambda_1 & \lambda_2 & \lambda_3 & \lambda_4 \\
\lambda_1^{2} & \lambda_2^{2} & \lambda_3^{2} & \lambda_4^{2} \\
\lambda_1^{3} & \lambda_2^{3} & \lambda_3^{3} & \lambda_4^{3}
\end{pmatrix}
\begin{pmatrix}
\lambda_{1j} \\
\lambda_{2j} \\
\lambda_{3j} \\
\lambda_{4j}
\end{pmatrix}
=
\begin{pmatrix}
0 \\
0 \\
0 \\
-K_j
\end{pmatrix},
\end{equation}
where $K_j$ is defined as follows:
$$dK = -\frac{df_4}{4} = -\sum_{j=1}^4 \left( \sum_{i=1}^4 \lambda_i^{3} \lambda_{ij} \right) \omega_j =: \sum_{j=1}^4 K_j \omega_j.$$
Denote the $4 \times 4$ Vandermonde matrix on the left-hand side of \eqref{Vandermonde matrix} by $D$. 
It is known that its determinant
$$
\gamma := \det D = \prod_{k,l=1; \; k > l}^4 (\lambda_k - \lambda_l) \neq 0.$$
Then it follows from the equation \eqref{Vandermonde matrix} that
\begin{equation}  \label{lambda-ij}
    \begin{aligned}
        \lambda_{ij} &= (-1)^{i+1} \frac{K_j}{\gamma} \prod_{k,l=1; \; k \neq i; k > l}^4 (\lambda_k - \lambda_l) \\
                     &= K_j \cdot \frac{1}{\prod_{k=1; k \neq i}^4 (\lambda_k - \lambda_i)}.
    \end{aligned}
\end{equation}

\subsection{The 3-form  \texorpdfstring{$\Phi$}{Phi}}
$\ $

\medskip

As in {\cite{Tongzhu Li}}, we define a 3-form $\Phi$ as follows:
$$\Phi = \sum_{i<j} (\lambda_i + \lambda_j)\theta_{ij},$$
where
$$\theta_{12}=\omega_3 \wedge \omega_4 \wedge \omega_{12},\ \  \theta_{13}=\omega_4 \wedge \omega_2 \wedge \omega_{13}, \ \  \theta_{14}=\omega_2 \wedge \omega_3 \wedge \omega_{14}, $$
$$\theta_{23}=\omega_1 \wedge \omega_4 \wedge \omega_{23},\ \ \theta_{24}=\omega_3 \wedge \omega_1 \wedge \omega_{24},\ \ \theta_{34}=\omega_1 \wedge \omega_2 \wedge \omega_{34}. $$
By {\cite{Tongzhu Li}}, we know that $\Phi$ is globally well defined on $Y$.

From {\cite{Tongzhu Li}}, we also have the differential of $\Phi$ as follows:
\begin{equation}  \label{d Phi}
    d\Phi = f_3(\sum_{i}c_i h_{44i}^{2} + 1) \cdot \text{vol} \quad \text{on}\ Y,
\end{equation}
where
$$c_i = \displaystyle\frac{2(3S-4\lambda_{i}^{2})}{3(\lambda_1 - \lambda_2)^{2}(\lambda_1 - \lambda_3)^{2}(\lambda_2 - \lambda_3)^{2}}.$$
By {\cite{Tongzhu Li}}, we know that $\sum_{i}c_i h_{44i}^{2} + 1 > 0$.
Then under our assumption $f_3 \geqslant  0$, we have
\begin{equation} \label{integral}
    \int_{Y} d\Phi \geqslant  0.
\end{equation}

Next, we are going to calculate $dK \wedge \Phi$.

From \eqref{w-ij} and \eqref{h-ijk}, it follows that
\begin{equation*} 
    \begin{aligned}
        dK \wedge \Phi = [&(\lambda_1 + \lambda_2)(K_1 \varGamma_{122} - K_2 \varGamma_{121} ) + (\lambda_1 + \lambda_3)(K_1 \varGamma_{133} - K_3 \varGamma_{131} ) \\
                       +&(\lambda_1 + \lambda_4)(K_1 \varGamma_{144} - K_4 \varGamma_{141} ) + (\lambda_2 + \lambda_3)(K_2 \varGamma_{233} - K_3 \varGamma_{232} ) \\
                       +&(\lambda_2 + \lambda_4)(K_2 \varGamma_{244} - K_4 \varGamma_{242} ) + (\lambda_3 + \lambda_4)(K_3 \varGamma_{344} - K_4 \varGamma_{343} ) ] \\
                       \cdot & \omega_1 \wedge \omega_2 \wedge \omega_3 \wedge \omega_4.
    \end{aligned}
\end{equation*}
By \eqref{h-iik}, \eqref{h-ijk} and \eqref{lambda-ij}, we have 
$$K_1 \varGamma_{122} - K_2 \varGamma_{121} = \displaystyle\frac{K_{1}^{2}}{(\lambda_1 - \lambda_2)^{2}(\lambda_3 - \lambda_2)(\lambda_4 - \lambda_2)} + \displaystyle\frac{K_{2}^{2}}{(\lambda_1 - \lambda_2)^{2}(\lambda_3 - \lambda_1)(\lambda_4 - \lambda_1)}, $$
$$K_1 \varGamma_{133} - K_3 \varGamma_{131} = \displaystyle\frac{K_{1}^{2}}{(\lambda_1 - \lambda_3)^{2}(\lambda_2 - \lambda_3)(\lambda_4 - \lambda_3)} + \displaystyle\frac{K_{3}^{2}}{(\lambda_1 - \lambda_3)^{2}(\lambda_2 - \lambda_1)(\lambda_4 - \lambda_1)}, $$
$$K_1 \varGamma_{144} - K_4 \varGamma_{141} = \displaystyle\frac{K_{1}^{2}}{(\lambda_1 - \lambda_4)^{2}(\lambda_2 - \lambda_4)(\lambda_3 - \lambda_4)} + \displaystyle\frac{K_{4}^{2}}{(\lambda_1 - \lambda_4)^{2}(\lambda_2 - \lambda_1)(\lambda_3 - \lambda_1)}, $$
$$K_2 \varGamma_{233} - K_3 \varGamma_{232} = \displaystyle\frac{K_{2}^{2}}{(\lambda_2 - \lambda_3)^{2}(\lambda_1 - \lambda_3)(\lambda_4 - \lambda_3)} + \displaystyle\frac{K_{3}^{2}}{(\lambda_2 - \lambda_3)^{2}(\lambda_1 - \lambda_2)(\lambda_4 - \lambda_2)}, $$
$$K_2 \varGamma_{244} - K_4 \varGamma_{242} = \displaystyle\frac{K_{2}^{2}}{(\lambda_2 - \lambda_4)^{2}(\lambda_1 - \lambda_4)(\lambda_3 - \lambda_4)} + \displaystyle\frac{K_{4}^{2}}{(\lambda_2 - \lambda_4)^{2}(\lambda_1 - \lambda_2)(\lambda_3 - \lambda_2)}, $$
$$K_3 \varGamma_{344} - K_4 \varGamma_{343} = \displaystyle\frac{K_{3}^{2}}{(\lambda_3 - \lambda_4)^{2}(\lambda_1 - \lambda_4)(\lambda_2 - \lambda_4)} + \displaystyle\frac{K_{4}^{2}}{(\lambda_3 - \lambda_4)^{2}(\lambda_1 - \lambda_3)(\lambda_2 - \lambda_3)}. $$
Therefore,
\begin{equation} \label{dK}
    dK \wedge \Phi = \sum_{i=1}^{4}u_i K_{i}^{2} \cdot \text{vol} \quad \text{on}\  Y,
\end{equation}
where
\begin{align*}u_1 
=& \displaystyle\frac{\lambda_1 + \lambda_2}{(\lambda_1 - \lambda_2)^{2}(\lambda_3 - \lambda_2)(\lambda_4 - \lambda_2)} + \displaystyle\frac{\lambda_1 + \lambda_3}{(\lambda_1 - \lambda_3)^{2}(\lambda_2 - \lambda_3)(\lambda_4 - \lambda_3)} \\
&+ \displaystyle\frac{\lambda_1 + \lambda_4}{(\lambda_1 - \lambda_4)^{2}(\lambda_2 - \lambda_4)(\lambda_3 - \lambda_4)}, 
\end{align*}

\begin{align*}
u_2 
=& \displaystyle\frac{\lambda_1 + \lambda_2}{(\lambda_1 - \lambda_2)^{2}(\lambda_3 - \lambda_1)(\lambda_4 - \lambda_1)} + \displaystyle\frac{\lambda_2 + \lambda_3}{(\lambda_2 - \lambda_3)^{2}(\lambda_1 - \lambda_3)(\lambda_4 - \lambda_3)} \\
&+ \displaystyle\frac{\lambda_2 + \lambda_4}{(\lambda_2 - \lambda_4)^{2}(\lambda_1 - \lambda_4)(\lambda_3 - \lambda_4)}, \end{align*}

\begin{align*} 
u_3 
=& \displaystyle\frac{\lambda_1 + \lambda_3}{(\lambda_1 - \lambda_3)^{2}(\lambda_2 - \lambda_1)(\lambda_4 - \lambda_1)} + \displaystyle\frac{\lambda_2 + \lambda_3}{(\lambda_2 - \lambda_3)^{2}(\lambda_1 - \lambda_2)(\lambda_4 - \lambda_2)}\\
&+ \displaystyle\frac{\lambda_3 + \lambda_4}{(\lambda_3 - \lambda_4)^{2}(\lambda_1 - \lambda_4)(\lambda_2 - \lambda_4)}, \end{align*}

\begin{align*}
u_4 
=& \displaystyle\frac{\lambda_1 + \lambda_4}{(\lambda_1 - \lambda_4)^{2}(\lambda_2 - \lambda_1)(\lambda_3 - \lambda_1)} + \displaystyle\frac{\lambda_2 + \lambda_4}{(\lambda_2 - \lambda_4)^{2}(\lambda_1 - \lambda_2)(\lambda_3 - \lambda_2)}\\ 
&+ \displaystyle\frac{\lambda_3 + \lambda_4}{(\lambda_3 - \lambda_4)^{2}(\lambda_1 - \lambda_3)(\lambda_2 - \lambda_3)}. \end{align*}

$ \ $

For $0 < \varepsilon < \frac{b - a}{2}$, we let
\begin{equation}
    \begin{aligned}
        X_{\varepsilon} &:= \{p \in M^{4}: a < K(p) < a + \varepsilon\}, \\
        Y_{\varepsilon} &:= \{p \in M^{4}: a + \varepsilon \leqslant K(p) \leqslant b - \varepsilon\}, \\
        Z_{\varepsilon} &:= \{p \in M^{4}: b - \varepsilon < K(p) < b\}
    \end{aligned}
\end{equation}
and infer from \eqref{set classification} that
$$Y = X_{\varepsilon} \cup Y_{\varepsilon} \cup Z_{\varepsilon}.$$

\begin{lemma}  \label{lem1}
    There exists a constant $A>0$ depending only on $S,f_3$, such that
    $$ u_i \geqslant  -A \quad \text{on } X_{\varepsilon} \quad and \quad u_i \leqslant   A \quad \text{on } Z_{\varepsilon}. $$
\end{lemma}

\begin{proof}
Let $K^{-1}(a)$ denote the set of $(\alpha _1,\alpha_2,\alpha _3,\alpha_4)\in \mathbb{R}^4$, where $\alpha_1\leqslant\alpha_2\leqslant\alpha_3\leqslant\alpha_4$ are principal curvatures at some point in $X$ such that $\prod_{i=1}^4\alpha_i=a$.

From previous discussions, when $K\rightarrow a$, it happens that
$$\alpha_1 < \alpha_2 = \alpha_3 < 0 < \alpha_4.$$

Obviously, $u_2,u_3\rightarrow +\infty $ as $K\rightarrow a$. 

In addition, by $\lambda_1 + \lambda_2 + \lambda_3 + \lambda_4 = 0$, we have
\begin{align*}
    u_1 =& \displaystyle\frac{\lambda_1 + \lambda_2}{(\lambda_1 - \lambda_2)^{2}(\lambda_3 - \lambda_2)(\lambda_4 - \lambda_2)} + \displaystyle\frac{\lambda_1 + \lambda_3}{(\lambda_1 - \lambda_3)^{2}(\lambda_2 - \lambda_3)(\lambda_4 - \lambda_3)} \\
    &+ \displaystyle\frac{\lambda_1 + \lambda_4}{(\lambda_1 - \lambda_4)^{2}(\lambda_2 - \lambda_4)(\lambda_3 - \lambda_4)} \\
        =& \displaystyle\frac{\lambda_4 + \lambda_3}{(\lambda_1 - \lambda_2)^{2}(\lambda_2 - \lambda_3)(\lambda_4 - \lambda_2)} - \displaystyle\frac{\lambda_4 + \lambda_2}{(\lambda_1 - \lambda_3)^{2}(\lambda_2 - \lambda_3)(\lambda_4 - \lambda_3)}\\
        &+ \displaystyle\frac{\lambda_1 + \lambda_4}{(\lambda_1 - \lambda_4)^{2}(\lambda_2 - \lambda_4)(\lambda_3 - \lambda_4)} \\[4pt]
        =& \displaystyle\frac{(\lambda_1 - \lambda_3)^{2}(\lambda_{4}^{2}-\lambda_{3}^{2}) - (\lambda_1 - \lambda_2)^{2}(\lambda_{4}^{2}-\lambda_{2}^{2})}{(\lambda_1 - \lambda_2)^{2}(\lambda_1 - \lambda_3)^{2}(\lambda_4 - \lambda_2)(\lambda_4 - \lambda_3)(\lambda_2 - \lambda_3)}  \\
        &+ \displaystyle\frac{\lambda_1 + \lambda_4}{(\lambda_1 - \lambda_4)^{2}(\lambda_2 - \lambda_4)(\lambda_3 - \lambda_4)} \\[4pt]
        =& \displaystyle\frac{(\lambda_2 + \lambda_3)(S-2\lambda_{4}^{2}) + 2\lambda_{1}(\lambda_{4}^{2}-\lambda_{2}^{2}-\lambda_{3}^{2}-\lambda_2 \lambda_3)}{(\lambda_1 - \lambda_2)^{2}(\lambda_1 - \lambda_3)^{2}(\lambda_4 - \lambda_2)(\lambda_4 - \lambda_3)}  \\
        &+ \displaystyle\frac{\lambda_1 + \lambda_4}{(\lambda_1 - \lambda_4)^{2}(\lambda_2 - \lambda_4)(\lambda_3 - \lambda_4)}, 
  \end{align*}

\begin{align*}
  u_4 =& \displaystyle\frac{\lambda_1 + \lambda_4}{(\lambda_1 - \lambda_4)^{2}(\lambda_2 - \lambda_1)(\lambda_3 - \lambda_1)} + \displaystyle\frac{\lambda_2 + \lambda_4}{(\lambda_2 - \lambda_4)^{2}(\lambda_1 - \lambda_2)(\lambda_3 - \lambda_2)} \\
       &+ \displaystyle\frac{\lambda_3 + \lambda_4}{(\lambda_3 - \lambda_4)^{2}(\lambda_1 - \lambda_3)(\lambda_2 - \lambda_3)} \\
      =& \displaystyle\frac{\lambda_1 + \lambda_4}{(\lambda_1 - \lambda_4)^{2}(\lambda_2 - \lambda_1)(\lambda_3 - \lambda_1)} + \displaystyle\frac{\lambda_1 + \lambda_3}{(\lambda_2 - \lambda_4)^{2}(\lambda_1 - \lambda_2)(\lambda_2 - \lambda_3)} \\
      &- \displaystyle\frac{\lambda_1 + \lambda_2}{(\lambda_3 - \lambda_4)^{2}(\lambda_1 - \lambda_3)(\lambda_2 - \lambda_3)} \\
      =& \displaystyle\frac{\lambda_1 + \lambda_4}{(\lambda_1 - \lambda_4)^{2}(\lambda_2 - \lambda_1)(\lambda_3 - \lambda_1)}\\[4pt]
      &+ \displaystyle\frac{(\lambda_3 - \lambda_4)^{2}(\lambda_{1}^{2}-\lambda_{3}^{2}) - (\lambda_2 - \lambda_4)^{2}(\lambda_{1}^{2}-\lambda_{2}^{2})}{(\lambda_2 - \lambda_4)^{2}(\lambda_3 - \lambda_4)^{2}(\lambda_1 - \lambda_2)(\lambda_1 - \lambda_3)(\lambda_2 - \lambda_3)}  \\
      =& \displaystyle\frac{\lambda_1 + \lambda_4}{(\lambda_1 - \lambda_4)^{2}(\lambda_2 - \lambda_1)(\lambda_3 - \lambda_1)} \\[4pt]
      &+ \displaystyle\frac{(\lambda_2 + \lambda_3)(S-2\lambda_{1}^{2}) + 2\lambda_{4}(\lambda_{1}^{2}-\lambda_{2}^{2}-\lambda_{3}^{2}-\lambda_2 \lambda_3)}{(\lambda_2 - \lambda_4)^{2}(\lambda_3 - \lambda_4)^{2}(\lambda_1 - \lambda_2)(\lambda_1 - \lambda_3)}. 
\end{align*}

It is known that when $K \to a$, we have 
  \begin{align*}
    u_1 &\to \frac{2\alpha_2 (S-2\alpha_{4}^{2}) + 2\alpha_{1}(\alpha_{4}^{2}-3\alpha_{2}^{2})}{(\alpha_1 - \alpha_2)^{4}(\alpha_4 - \alpha_2)^{2}} 
        + \frac{\alpha_1 + \alpha_4}{(\alpha_1 - \alpha_4)^{2}(\alpha_2 - \alpha_4)^{2}}, \\
    u_4 &\to \frac{\alpha_1 + \alpha_4}{(\alpha_1 - \alpha_4)^{2}(\alpha_2 - \alpha_1)^{2}} 
        + \frac{2\alpha_2 (S-2\alpha_{1}^{2}) + 2\alpha_{4}(\alpha_{1}^{2}-3\alpha_{2}^{2})}{(\alpha_2 - \alpha_4)^{4}(\alpha_1 - \alpha_2)^{2}}.
  \end{align*}
So, $u_1,u_4$ are bounded as $K\rightarrow a$. 
Therefore, there exists a constant $A_1 >0$, such that $ u_i \geqslant -A_1 $ on $X_{\varepsilon}$.

Similarly, let $K^{-1}(b)$ denote the set of $(\beta_1,\beta_2,\beta_3,\beta_4)\in \mathbb{R}^4$, where $\beta_1\leqslant\beta_2\leqslant\beta_3\leqslant\beta_4$ are principal curvatures at some point in $Z$ such that $\prod_{i=1}^4\beta_i=b$.

From previous discussions, when $K\rightarrow b$ from below, it only happens that
$$\beta_1 = \beta_2 < \beta_3 < \beta_4,\ \ \beta_1 = \beta_2 < 0.$$

Obviously, $u_1,u_2\rightarrow -\infty $ as $K\rightarrow b$. 

In addition, by $\lambda_1 + \lambda_2 + \lambda_3 + \lambda_4 = 0$, we have
\begin{align*}
    u_3 =& \displaystyle\frac{\lambda_1 + \lambda_3}{(\lambda_1 - \lambda_3)^{2}(\lambda_2 - \lambda_1)(\lambda_4 - \lambda_1)} + \displaystyle\frac{\lambda_2 + \lambda_3}{(\lambda_2 - \lambda_3)^{2}(\lambda_1 - \lambda_2)(\lambda_4 - \lambda_2)} \\
    &+ \displaystyle\frac{\lambda_3 + \lambda_4}{(\lambda_3 - \lambda_4)^{2}(\lambda_1 - \lambda_4)(\lambda_2 - \lambda_4)} \\
        =& \displaystyle\frac{\lambda_4 + \lambda_2}{(\lambda_1 - \lambda_3)^{2}(\lambda_1 - \lambda_2)(\lambda_4 - \lambda_1)} - \displaystyle\frac{\lambda_4 + \lambda_1}{(\lambda_2 - \lambda_3)^{2}(\lambda_1 - \lambda_2)(\lambda_4 - \lambda_2)}\\
        &+ \displaystyle\frac{\lambda_3 + \lambda_4}{(\lambda_3 - \lambda_4)^{2}(\lambda_1 - \lambda_4)(\lambda_2 - \lambda_4)} \\[4pt]
        =& \displaystyle\frac{(\lambda_2 - \lambda_3)^{2}(\lambda_{4}^{2}-\lambda_{2}^{2}) - (\lambda_1 - \lambda_3)^{2}(\lambda_{4}^{2}-\lambda_{1}^{2})}{(\lambda_1 - \lambda_3)^{2}(\lambda_2 - \lambda_3)^{2}(\lambda_4 - \lambda_1)(\lambda_4 - \lambda_2)(\lambda_1 - \lambda_2)} \\
        &+ \displaystyle\frac{\lambda_3 + \lambda_4}{(\lambda_3 - \lambda_4)^{2}(\lambda_1 - \lambda_4)(\lambda_2 - \lambda_4)} \\[4pt]
        =& \displaystyle\frac{(\lambda_1 + \lambda_2)(S-2\lambda_{4}^{2}) + 2\lambda_{3}(\lambda_{4}^{2}-\lambda_{1}^{2}-\lambda_{2}^{2}-\lambda_1 \lambda_2)}{(\lambda_1 - \lambda_3)^{2}(\lambda_2 - \lambda_3)^{2}(\lambda_4 - \lambda_1)(\lambda_4 - \lambda_2)}  \\
        &+ \displaystyle\frac{\lambda_3 + \lambda_4}{(\lambda_3 - \lambda_4)^{2}(\lambda_1 - \lambda_4)(\lambda_2 - \lambda_4)},
\end{align*}

\begin{align*}
    u_4 =& \displaystyle\frac{\lambda_1 + \lambda_4}{(\lambda_1 - \lambda_4)^{2}(\lambda_2 - \lambda_1)(\lambda_3 - \lambda_1)} + \displaystyle\frac{\lambda_2 + \lambda_4}{(\lambda_2 - \lambda_4)^{2}(\lambda_1 - \lambda_2)(\lambda_3 - \lambda_2)} \\
    &+ \displaystyle\frac{\lambda_3 + \lambda_4}{(\lambda_3 - \lambda_4)^{2}(\lambda_1 - \lambda_3)(\lambda_2 - \lambda_3)} \\
        =& \displaystyle\frac{\lambda_3 + \lambda_2}{(\lambda_1 - \lambda_4)^{2}(\lambda_1 - \lambda_2)(\lambda_3 - \lambda_1)} - \displaystyle\frac{\lambda_3 + \lambda_1}{(\lambda_2 - \lambda_4)^{2}(\lambda_1 - \lambda_2)(\lambda_3 - \lambda_2)}\\
        &+ \displaystyle\frac{\lambda_3 + \lambda_4}{(\lambda_3 - \lambda_4)^{2}(\lambda_1 - \lambda_3)(\lambda_2 - \lambda_3)} \\[4pt]
        =& \displaystyle\frac{(\lambda_2 - \lambda_4)^{2}(\lambda_{3}^{2}-\lambda_{2}^{2}) - (\lambda_1 - \lambda_4)^{2}(\lambda_{3}^{2}-\lambda_{1}^{2})}{(\lambda_1 - \lambda_4)^{2}(\lambda_2 - \lambda_4)^{2}(\lambda_3 - \lambda_1)(\lambda_3 - \lambda_2)(\lambda_1 - \lambda_2)}  \\
        &+ \displaystyle\frac{\lambda_3 + \lambda_4}{(\lambda_3 - \lambda_4)^{2}(\lambda_1 - \lambda_3)(\lambda_2 - \lambda_3)} \\[4pt]
        =& \displaystyle\frac{(\lambda_1 + \lambda_2)(S-2\lambda_{3}^{2}) + 2\lambda_{4}(\lambda_{3}^{2}-\lambda_{1}^{2}-\lambda_{2}^{2}-\lambda_1 \lambda_2)}{(\lambda_1 - \lambda_4)^{2}(\lambda_2 - \lambda_4)^{2}(\lambda_3 - \lambda_1)(\lambda_3 - \lambda_2)}  \\
        &+ \displaystyle\frac{\lambda_3 + \lambda_4}{(\lambda_3 - \lambda_4)^{2}(\lambda_1 - \lambda_3)(\lambda_2 - \lambda_3)}.
\end{align*}

It is known that when $K \to b$, we have 
\begin{align*}
  u_3 &\to \frac{2\beta_1 (S-2\beta_{4}^{2}) + 2\beta_{3}(\beta_{4}^{2}-3\beta_{1}^{2})}{(\beta_1 - \beta_3)^{4}(\beta_1 - \beta_4)^{2}} 
      + \frac{\beta_3 + \beta_4}{(\beta_3 - \beta_4)^{2}(\beta_1 - \beta_4)^{2}}, \\[4pt]
  u_4 &\to \frac{2\beta_1 (S-2\beta_{3}^{2}) + 2\beta_{4}(\beta_{3}^{2}-3\beta_{1}^{2})}{(\beta_1 - \beta_4)^{4}(\beta_1 - \beta_3)^{2}} 
      + \frac{\beta_3 + \beta_4}{(\beta_3 - \beta_4)^{2}(\beta_1 - \beta_3)^{2}}.
\end{align*}
So, $u_3,u_4$ is a finite value as $K\rightarrow b$. 
Therefore, there exists a constant $A_2 >0$, such that $ u_i \leqslant  A_2 $ on $Z_{\varepsilon}$. Then we take
$$A := \max\{A_1, A_2\},$$
and the lemma follows.
\end{proof}

\subsection{Proof of Theorem \ref{mthm}}

$\ \ $

\medskip

As we remarked before $Y \neq \varnothing $, for any smooth function $\eta : (a, b) \to \mathbb{R}$ with compact support, we apply Stokes' theorem to
$$d((\eta \circ K)\Phi) = (\eta \circ K)d\Phi + (\eta' \circ K)dK \wedge \Phi$$
to obtain
\begin{equation}  \label{Stokes1}
    \int_Y (\eta \circ K)d\Phi + \int_Y (\eta' \circ K)dK \wedge \Phi = 0.
\end{equation}

Given a small positive $\varepsilon$, we choose a smooth function $\eta_\varepsilon : \mathbb{R} \to \mathbb{R}$ such that

(1) $0 \leqslant \eta_\varepsilon \leqslant 1$;

(2) $\eta_\varepsilon(t) = 0$ for $t \leqslant a + \frac{\varepsilon}{3} $ or $t \geqslant b - \frac{\varepsilon}{3} $;

(3) $\eta_\varepsilon(t) = 1$ for $a + \varepsilon \leqslant t \leqslant b - \varepsilon$;

(4) $\eta'_\varepsilon \geqslant 0$ on $(a + \frac{\varepsilon}{3}, a + \varepsilon)$ and $\eta'_\varepsilon \leqslant 0$ on $(b - \varepsilon, b - \frac{\varepsilon}{3})$.

It follows from \eqref{integral}, \eqref{dK}, \eqref{Stokes1} and Lemma \ref{lem1} that
\begin{equation}  \label{Scaling inequality}
    \begin{aligned}
        0 &\leqslant \int_Y (\eta_\varepsilon \circ K)d\Phi \\
          &= -\int_Y (\eta'_\varepsilon \circ K)dK \wedge \Phi \\
          &= -\int_Y (\eta'_\varepsilon \circ K)\sum_{i=1}^4 u_iK_i^2 \cdot \text{vol} \\
          &\leqslant A \cdot \int_Y |\eta'_\varepsilon \circ K| \cdot |dK|^2 \cdot \text{vol},
    \end{aligned}
\end{equation}
where the second inequality is because Lemma \ref{lem1} and conditions (3) and (4) implying
$$-(\eta'_\varepsilon \circ K)\sum_{i=1}^4 u_iK_i^2 \leqslant A \cdot |\eta'_\varepsilon \circ K| \cdot |dK|^2 \quad \text{on} \ X_{\varepsilon} \cup Z_{\varepsilon}$$
and 
$$-(\eta'_\varepsilon \circ K)\sum_{i=1}^4 u_iK_i^2 = 0 \quad \text{on} \ Y_{\varepsilon}.$$

On the other hand, for any smooth function $\xi: \mathbb{R} \to \mathbb{R}$, we may apply Stokes' theorem to
$$d^* ((\xi \circ K) dK) = (\xi' \circ K) |dK|^2 \cdot \mathrm{vol} + (\xi \circ K) \Delta K \cdot \mathrm{vol}$$
to obtain
\begin{equation}  \label{Stokes2}
    \int_{M^4} (\xi' \circ K) |dK|^2 \cdot \mathrm{vol} + \int_{M^4} (\xi \circ K) \Delta K \cdot \mathrm{vol} = 0.
\end{equation}

Let $\xi_\varepsilon: \mathbb{R} \to \mathbb{R}$ be the smooth function given by
$$\xi_\varepsilon = 
\begin{cases} 
    \eta_\varepsilon - 1 & \text{on } \left( -\infty, \frac{a+b}{2} \right], \\ 
    1 - \eta_\varepsilon & \text{on } \left[ \frac{a+b}{2}, +\infty \right). 
\end{cases}$$

Note that $\xi_\varepsilon' = |\eta_\varepsilon'|$. It follows from \eqref{Stokes2} that
\begin{align*}
    \int_Y |\eta_\varepsilon' \circ K| \cdot |dK|^2 \cdot \mathrm{vol} 
    & = - \int_{M^4} (\xi_\varepsilon \circ K) \Delta K \cdot \mathrm{vol} \\
    & \leqslant \int_{M^4} |\xi_\varepsilon \circ K| \cdot |\Delta K| \cdot \mathrm{vol}.
\end{align*}

By construction $|\xi_\varepsilon| \leqslant 1$ and $\xi_\varepsilon \circ K = 0$ on $Y_\varepsilon$.

Next, we need the following lemma which was proved in \cite{de Almeida} for $n=3$ and still holds for $n=4$.

\begin{lemma} \label{lemma of de Almeida}
    Let $u: M^4 \to \mathbb{R}$ be a smooth function and $m= \min_{M^4}u$. If $D_{\varepsilon}=u^{-1}([m,m+\varepsilon])$, then
    $$\lim_{\varepsilon \to 0} \int_{D_{\varepsilon}} |\Delta u| \mathrm{vol} = 0.$$
    In particular, 
    $$\lim_{\varepsilon \to 0} \int_{M^4 - Y_\varepsilon} |\Delta K| \cdot \mathrm{vol} = 0, \quad \text{if } X \cup Z \neq \varnothing .$$
\end{lemma}

Then due to Lemma \ref{lemma of de Almeida}, we obtain
\begin{align*}
  \lim_{\varepsilon \to 0} \int_{M^4} |\xi_\varepsilon \circ K| \cdot |\Delta K| \cdot \mathrm{vol} 
  &= \lim_{\varepsilon \to 0} \int_{M^4 - Y_\varepsilon} |\xi_\varepsilon \circ K| \cdot |\Delta K| \cdot \mathrm{vol} \\
  &\leqslant \lim_{\varepsilon \to 0} \int_{M^4 - Y_\varepsilon} |\Delta K| \cdot \mathrm{vol} \\
  &= 0,
\end{align*}
and thus,
$$\lim_{\varepsilon\to0}\int_{Y}|\eta_{\varepsilon}' \circ K|\cdot|dK|^{2}\cdot\mathrm{vol}=0.$$
Combining \eqref{Scaling inequality}, we have
$$
\lim_{\varepsilon\to0}\int_{Y}(\eta_{\varepsilon}\circ K)d\Phi=0.$$

At last, \eqref{d Phi} leads to
$$0\leqslant\int_{Y_{\varepsilon^{\prime}}}f_3(\sum_{i}c_i h_{44i}^{2} + 1) \cdot \text{vol}\leqslant\int_{Y}(\eta_{\varepsilon}\circ K)d\Phi$$
for all $0<\varepsilon\leqslant\varepsilon^{\prime}<\frac{b-a}{2}$, it follows that $f_3=0$ on $Y$. Furthermore, $f_3=0$ on $M^4$. 

Thus, by Theorem \ref{thm-Deng}, the proof of the main theorem is now complete.

\section*{Acknowledgments}

    This research was supported by the National Natural Science Foundation of China, Grant Nos. 12471051, 12171423, and 12071424.

\end{document}